\documentclass{amsart}

\usepackage{diagbox}

\usepackage{stmaryrd}

\usepackage{graphicx} 
\usepackage{amsmath,amsfonts,euscript,amscd,amsthm,amssymb,upref,graphics,color,verbatim, float, placeins,mathrsfs, mathtools,comment,tikz, multirow, makecell, tikz-cd}

\newcommand{\Bscr}{\mathscr{B}}  

\newcommand{\Dscr}{\mathscr{D}}

\newcommand{\Lscr}{\mathscr{L}}
\newcommand{\Mscr}{\mathscr{M}}

\newcommand{\Oscr}{\mathscr{O}}
\newcommand{\Pscr}{\mathscr{P}}

\newcommand{\Xscr}{\mathscr{X}}

\usepackage{ulem}  

\usepackage[margin=3cm]{geometry}
\usepackage[all]{xy}
\usepackage{hyperref}
\usepackage[shortlabels]{enumitem}

\newcommand{\lnd}{\operatorname{{\rm LND}}}

\newcommand{\Der}{\operatorname{{\rm Der}}}

\newcommand{\Nat}{\ensuremath{\mathbb{N}}}

\newcommand{\aff}{\ensuremath{\mathbb{A}}}
\newcommand{\bk}{{\ensuremath{\rm \bf k}}}

\newcommand{\trdeg}{	\operatorname{{\rm trdeg}}}
\newcommand{\Frac}{		\operatorname{{\rm Frac}}}

\newcommand{\Lie}{		\operatorname{{\rm Lie}}}

\newcommand{\setspec}[2]{\big\{\,#1\, \mid \,#2\, \big\}}

\newcommand{\isom}{\cong}

\newcommand{\Aut}{		\operatorname{{\rm Aut}}}

\newcommand{\Spec}{		\operatorname{{\rm Spec}}}

\newcommand{\locfin}{\operatorname{{\rm locfin}}}

\DeclareMathOperator{\image}{im}

\newtheorem{corollary}{Corollary}

\swapnumbers

\newtheorem*{theorem*}{Theorem}
\newtheorem{proposition}[subsection]{Proposition}
\newtheorem*{proposition*}{Proposition}
\newtheorem{lemma}[subsection]{Lemma}

\newtheorem*{lemma*}{Lemma}

\theoremstyle{definition}

\newtheorem{remark}[subsection]{Remark}
\newtheorem*{remarkno}{Remark}
\newtheorem{definition}[subsection]{Definition}

\newtheorem{nothing*}[subsection]{}

\newtheorem{remarks}[subsection]{Remarks}

\newcommand{\pgoth}{\mathfrak{p}}

\newcommand{\Leul}{\EuScript{L}}

\title{Locally finite sets of derivations}

\author{Michael Chitayat, Daniel Daigle and Andriy Regeta}

\address{Department of Mathematics and Statistics, University of Ottawa, Ottawa, ON, Canada, K1N 6N5.}
\email{ddaigle@uottawa.ca}

\address{Dipartimento di Matematica ``Tullio Levi-Civita'', 
Universit\`a di Padova, Via Trieste 63, I-35121 Padova}
\email{michael.chitayat@unipd.it, andriyregeta@gmail.com}

\date{\today}

\begin{document}



\newcommand{\I}{\text{\rm I}}
\newcommand{\II}{\text{\rm II}}
\newcommand{\plinth}{\operatorname{{\rm pl}}}

\newcommand{\bbD}{\mathbb{D}}
\newcommand{\dgoth}{\mathfrak{d}}
\newcommand{\Hom}{\operatorname{{\rm Hom}}}
\newcommand{\Span}{\operatorname{{\rm Span}}}
\newcommand{\kk}[1]{\bk^{[{#1}]}}
\newcommand{\Norm}{\operatorname{{N}}}
\newcommand{\ad}{	\operatorname{\text{\rm ad}}}

\newcommand{\UnitalAlg}[1]{\langle #1 \rangle_{\text{\rm u-alg}}}
\newcommand{\AssAlg}[1]{\langle #1 \rangle_{\text{\rm alg}}}
\newcommand{\LieAlg}[1]{\langle #1 \rangle_{\text{\rm Lie}}}

\newcommand{\lfd}{\operatorname{{\rm LFD}}}
\newcommand{\LieWLF}{\operatorname{{\rm LieWLF}}}
\newcommand{\SolLieWLF}{\operatorname{{\rm SolLieWLF}}}

\newcommand{\rtrdeg}{	\operatorname{{\rm rtrdeg}}}
\newcommand{\grank}{	\operatorname{\rm grank}}

\newcommand{\rien}[1]{}

\newcommand{\ba}{{\bf a}}
\newcommand{\bb}{{\bf b}}
\newcommand{\bx}{{\bf x}}


\setlength{\marginparwidth}{2.5cm}  
\newcommand{\andriy}[1]{{\color{blue}*}\marginpar{\tiny \color{blue} AR: #1}}


\begin{abstract}
Given an algebra $B$ over a field $\bk$, we study conditions under which a Lie subalgebra of $\Der_\bk(B)$ is locally finite as a set of derivations.
As an application of our results, we show that if $X$ is a quasi-affine variety over an arbitrary field $\bk$,
and if $\Lscr$ is a finitely generated solvable Lie subalgebra of $\Der_\bk\Oscr_X(X)$ consisting of locally finite derivations,
then $\Lscr$ is locally finite. If, moreover, $\bk$ is algebraically closed and of characteristic zero, and $X$ is irreducible and affine, then $\Lscr$ is integrable.
\end{abstract}

\rien{ 
\begin{enumerate}

\item
Given an algebra $B$ over a field $\bk$,
we study conditions under which a Lie subalgebra of $\Der_\bk(B)$ is a locally finite set of derivations.
A consequence of our results is that if $X$ is a quasi-affine variety over an arbitrary field $\bk$,
and if $\Lscr$ is a solvable Lie subalgebra of $\Der_\bk\Oscr_X(X)$ consisting of locally finite derivations,
then every finitely generated Lie subalgebra of $\Lscr$ is locally finite.

\item Given an algebra $B$ over a field $\bk$,
we study conditions under which a Lie subalgebra of $\Der_\bk(B)$ is a locally finite set of derivations.
A consequence of our results is that if $X$ is a quasi-affine variety over an arbitrary field $\bk$,
and if $\Lscr$ is a solvable Lie subalgebra of $\Der_\bk\Oscr_X(X)$ consisting of locally finite derivations,
then $\Lscr$ is ``locally integrable'' in the sense that it admits a countable increasing filtration by locally finite Lie subalgebras.

\item Given an algebra $B$ over a field $\bk$,
we study conditions under which a Lie subalgebra of $\Der_\bk(B)$ is a locally finite set of derivations.
A consequence of our results is that if $X$ is a quasi-affine variety over an arbitrary field $\bk$,
and if $\Lscr$ is a solvable Lie subalgebra of $\Der_\bk\Oscr_X(X)$ consisting of locally finite derivations,
then $\Lscr$ admits a countable increasing filtration by locally finite Lie subalgebras.

\item Given an algebra $B$ over a field $\bk$,
we study conditions under which a Lie subalgebra of $\Der_\bk(B)$ is a locally finite set of derivations.
A consequence of our results is that if $X$ is a quasi-affine variety over an arbitrary field $\bk$,
and if $\Lscr$ is a finitely generated solvable Lie subalgebra of $\Der_\bk\Oscr_X(X)$ consisting of locally finite derivations,
then $\Lscr$ is a locally finite set of derivations.
If $\bk$ is algebraically closed and $X$ is irreducible and affine, $\Lscr$ is integrable.

\item Given an algebra $B$ over a field $\bk$, we study conditions under which a Lie subalgebra of $\Der_\bk(B)$ is locally finite as a set of derivations.
As an application of our results, we show that if $X$ is a quasi-affine variety over an arbitrary field $\bk$,
and if $\Lscr$ is a finitely generated solvable Lie subalgebra of $\Der_\bk\Oscr_X(X)$ consisting of locally finite derivations,
then $\Lscr$ is locally finite. If, moreover, $\bk$ is algebraically closed and $X$ is irreducible and affine, then $\Lscr$ is integrable.

\end{enumerate}
\end{abstract} 
} 


\maketitle


\section{Introduction}

Let $B$ be an algebra over a field $\bk$ and let $\Der_\bk(B)$ be the Lie algebra of $\bk$-derivations of $B$.
A subset $\Delta$ of $\Der_\bk(B)$ is called {\it locally finite\/} if for each $b \in B$
there exists a finite dimensional $\bk$-linear subspace $U$ of $B$ such that $b \in U$ and $D(U) \subseteq U$ for all $D \in \Delta$.
We say that $\Delta$ is {\it weakly locally finite\/} if every finite subset of $\Delta$ is locally finite.
It is straightforward to verify that if
$\Delta$ is locally finite (resp. weakly locally finite) then so is the Lie subalgebra $\LieAlg{\Delta}$ of $\Der_\bk(B)$ generated by $\Delta$.
A derivation $D \in \Der_\bk(B)$ is {\it locally finite\/} if $\{D\}$ is a locally finite subset of $\Der_\bk(B)$ in the above sense.
We denote the set of locally finite $\bk$-derivations of $B$ by $\lfd_\bk(B)$.

The concept of weak local finiteness is relevant to the study of automorphism groups of affine algebraic varieties,
but the terminology around it has not yet stabilized in the literature devoted to that subject.
In several sources the concept remains unnamed, with weakly locally finite Lie algebras described ad hoc as
``algebras that can be filtered by locally finite subalgebras.''
\cite{Arz_Zaid_2026} introduced the term ``locally integrable'' for what we call ``weakly locally finite''
(see Remark \ref{AE2F4C43-C9A8-4A99-805E-F255752D6D30} for details),
but since we develop the theory at a level of generality where integrability does not have a meaning,
it could be misleading to use that term here.
To ensure clarity and avoid cumbersome phrasing, we use the neutral term ``weakly locally finite.''

Clearly, if a Lie subalgebra $\Lscr$ of $\Der_\bk(B)$ is weakly locally finite then $\Lscr \subseteq \lfd_\bk(B)$.
This paper is motivated by the question of whether the converse holds,
and more generally by the problem of determining which hypotheses on a Lie subalgebra $\Lscr \subseteq \Der_\bk(B)$
ensure that $\Lscr$ is weakly locally finite.


Variants of this question are investigated in \cite{Zaid:LocFin_arXiv_2026} in the case where $B$ is the ring of regular functions on an affine variety.
In particular, \cite{Zaid:LocFin_arXiv_2026} shows that if $\bk$ is an algebraically closed field of characteristic zero
and $\Lscr$ is a solvable Lie subalgebra of $\Lie(\Aut \aff^2)$ generated by locally finite derivations, then $\Lscr$ is weakly locally finite.

We consider the above question from a very general standpoint.
In the present text, an {\it algebra} over a field $\bk$ is a pair $(B,\cdot)$ where $B$ is a $\bk$-vector space and ``$\cdot$'' is
an arbitrary $\bk$-bilinear map $B \times B \to B$, $(x,y) \mapsto x \cdot y$ (so $B$ need not be commutative or associative).
We say that the $\bk$-algebra $B$ is {\it derivation-finite\/} if there exists a finite subset $S$ of $B$ such that
the only $D \in \Der_\bk(B)$ that satisfies $D(x)=0$ for all $x \in S$ is the zero derivation.
Note that the algebras commonly encountered in algebraic geometry are derivation-finite (see \ref{E846360A-FB67-4DC0-8E96-92430F4D1817}).

The following are the main results of this note.
In the first four statements, $\bk$ is an arbitrary field and the only assumption on $B$ is that it is a derivation-finite $\bk$-algebra.

\begin{theorem*} 
Let $B$ be a derivation-finite algebra over a field $\bk$.
Let $S$ be a subset of $\lfd_\bk(B)$ that is closed under the Lie bracket and define the sequence of sets $(S_i)_{i \in \Nat}$
by declaring that $S_0 = S$ and that $S_{n+1} = \setspec{ [D,E] }{ D,E \in S_n }$ for each $n \in \Nat$.
\begin{enumerate}[\rm(a)]

\item  $S_0 \supseteq S_1 \supseteq S_2 \supseteq \cdots$ and each set $S_n$ is closed under the Lie bracket.

\item If there exists $n \in \Nat$ such that $S_n$ is a weakly locally finite set, then $S$ is a weakly locally finite set.

\end{enumerate}
\end{theorem*}

Given a Lie algebra $\Lscr$, let $\Lscr'$ be the Lie subalgebra of $\Lscr$ generated by $\setspec{ [x,y] }{ x,y \in \Lscr }$,
and define $\Lscr^{(0)} = \Lscr$ and $\Lscr^{(n+1)} = \big( \Lscr^{(n)} \big)'$ for $n\ge0$.
Recall that $\Lscr$ is {\it solvable\/} if there exists $d \in \Nat$ such that $\Lscr^{(d)} = 0$.

\begin{corollary} \label {21AFB0A0-0700-4E7F-8AE2-D3557F9656CC}
Let $B$ be a derivation-finite algebra over a field $\bk$
and let $\Lscr$ be a Lie subalgebra of $\Der_\bk(B)$ such that $\Lscr \subseteq \lfd_\bk(B)$.
\begin{enumerate}[\rm(a)]

\item If there exists $n \in \Nat$ such that $\Lscr^{(n)}$ is weakly locally finite, then $\Lscr$ is weakly locally finite.

\item If $\Lscr$ is solvable then it is weakly locally finite.

\end{enumerate}
\end{corollary}

In the special case where $B$ is the commutative polynomial ring in two variables over an algebraically closed field of characteristic zero,
part (b) of Corollary \ref{21AFB0A0-0700-4E7F-8AE2-D3557F9656CC} is also proved in \cite{Zaid:LocFin_arXiv_2026}.

\begin{corollary} \label {CFBD926C-B48D-4238-8595-F09AED548518}
Let $B$ be a derivation-finite algebra over a field $\bk$ and let $\Lscr$ be a solvable Lie subalgebra of $\Der_\bk(B)$.
If $\Lscr$ has a generating set included in $\lfd_\bk(B)$ and closed under the Lie bracket,
then $\Lscr$ is weakly locally finite.
\end{corollary}

\begin{corollary} \label {CA556DEE-FE7D-4364-A217-A983EBE59700}
Let $B$ be a derivation-finite algebra over a field $\bk$
and let $\Lscr$ be a Lie subalgebra of $\Der_\bk(B)$.
If $\Lscr$ is weakly locally finite
and $\Delta$ is a subset of $\lfd_\bk(B)$ such that $[\Delta,\Delta] \subseteq \Lscr$ and $[\Delta,\Lscr] \subseteq \Lscr$,
then  $\Mscr := \LieAlg{\Delta \cup \Lscr}$ is weakly locally finite and $[\Mscr,\Mscr] \subseteq \Lscr$. 
\end{corollary}

In the following statement, $\bk$ is an arbitrary field and $X$ is not necessarily irreducible:

\begin{corollary} \label {8B9F8C8D-5E7E-44EA-8762-9426050B1F90}
Let $X$ be a quasi-affine variety over a field $\bk$.
If $\Lscr$ is a solvable Lie subalgebra of $\Der_\bk\Oscr_X(X)$ consisting of locally finite derivations, then $\Lscr$ is weakly locally finite
and admits a countable increasing filtration by locally finite Lie subalgebras.
\end{corollary}

\begin{remarkno}
If $\bk$ is an algebraically closed field of characteristic zero and $X$ is an irreducible affine $\bk$-variety then the
conclusion of Corollary \ref{8B9F8C8D-5E7E-44EA-8762-9426050B1F90} is that $\Lscr$ is locally integrable, in the sense of \cite[Def.\ 2.5]{Arz_Zaid_2026}.
See Remark \ref{AE2F4C43-C9A8-4A99-805E-F255752D6D30} for details.
\end{remarkno}

\medskip

The results of this note will be used in a paper by the authors, in preparation,
on Borel subgroups of automorphism groups of affine varieties and Lie algebras of locally nilpotent derivations.

\medskip

We thank Mikhail Zaidenberg for valuable discussions that helped shape this note.

\section{Weakly locally finite sets of linear endomorphisms}

If $V$ is a vector space over a field $\bk$, we write $\Leul(V)$ for the associative $\bk$-algebra whose elements are the linear maps $F : V \to V$ and whose multiplication is the composition of maps.
Note that $\Leul(V)$ is also a Lie algebra over $\bk$, with $[F,G] = F \circ G - G \circ F$ for all $F,G \in \Leul(V)$.

\begin{definition} \label {582852DA-9D57-47AC-B351-A3A2CC446141}
Let $V$ be a vector space over a field $\bk$ and let $\Delta$ be a subset of $\Leul(V)$.
\begin{enumerate}[\rm(a)]

\item The set $\Delta$ is {\it locally finite\/} if for each $x \in V$ there exists a finite dimensional subspace $U$ of $V$
such that $x \in U$ and $F(U) \subseteq U$ for all $F \in \Delta$.

\item The set $\Delta$ is {\it weakly locally finite\/} if every finite subset of $\Delta$ is locally finite.

\end{enumerate}
\end{definition}

Clearly, if $\Delta$ is locally finite (resp.\ weakly locally finite) then so is every subset of $\Delta$.
We say that an element $F \in \Leul(V)$ is {\it locally finite\/} if $\{F\}$ is a locally finite subset of $\Leul(V)$ in the above sense.

\begin{remark} \label {03442952-E213-4085-B07A-FABC7408E85A}
Let $V$ be a vector space over a field $\bk$ and $\Delta$ a subset of $\Leul(V)$.
We write $\AssAlg{\Delta}$ for the {\it subalgebra of $\Leul(V)$ generated by $\Delta$}, i.e.,
the intersection of all $\bk$-subspaces $A$ of $\Leul(V)$ that are closed under composition of maps
and satisfy $\Delta \subseteq A$.
We write $\LieAlg{\Delta}$ for the {\it Lie subalgebra of $\Leul(V)$ generated by $\Delta$}.
It is easy to see that if $\Delta$ is locally finite (resp.\ weakly locally finite)
then so is each one of $\Span_\bk(\Delta) \subseteq \LieAlg{\Delta} \subseteq \AssAlg{\Delta}$.
\end{remark}

If $X,Y$ are subsets of a Lie algebra $\Lscr$, we define $[X,Y] = \setspec{ [x,y] }{ x \in X \text{ and } y \in Y }$, which is
also a subset of $\Lscr$.
The following result and Lemma \ref{104779FE-B0E2-4AD0-BB60-6BC97D307EB9} are analogous to \cite[Lemma 2]{Skutin_LND_2021}.

\begin{lemma}  \label {BB8F3187-088A-4A77-9DA8-3CF3D585A95B}
Let $V$ be a vector space over a field $\bk$ and let $\Delta$ be a locally finite subset of $\Leul(V)$.
If $T$ is a locally finite element of $\Leul(V)$ such that $[T,\Delta] \subseteq \Delta$, then $\{T\} \cup \Delta$ is locally finite.
\end{lemma}

\begin{proof}
Let $x \in V$. Since $\Delta$ is locally finite, there exists a finite dimensional subspace $U$ of $V$ such
that $x \in U$ and $F(U) \subseteq U$ for all $F \in \Delta$. Consider the subspace $W = \sum_{j=0}^\infty T^j(U)$ of $V$ and note that $x \in W$.
Since $U$ is finite dimensional and $T$ is locally finite, $W$ is finite dimensional.
Clearly, $T(W) \subseteq W$. To finish the proof, it suffices to show that $F(W) \subseteq W$ for all $F \in \Delta$.

For each $j \in \Nat$, define $W_j = U + T U + T^2 U + \cdots + T^j U$ and let $\Pscr(j)$ be the assertion
$$
F T^j U \subseteq W_j \qquad \text{for all $F \in \Delta$.}
$$
Since $F(U) \subseteq U$ for all $F \in \Delta$, $\Pscr(0)$ is true.
Suppose that $j \in \Nat$ is such that $\Pscr(j)$ is true.
Given any $F \in \Delta$ we have $FT^{j+1} = (F T)T^j = (T F - [T,F])T^j = T F T^j - F'T^j$ where $F' = [T,F] \in \Delta$, so $\Pscr(j)$ gives
$$
F T^{j+1} U
\subseteq T F T^j U + F'T^j U
\subseteq T W_j + W_j \subseteq W_{j+1},
$$
which shows that $\Pscr(j+1)$ is true. By induction, we obtain that $\Pscr(j)$ is true for all $j \in \Nat$.
It follows that if $F \in \Delta$ then $F(W) = \sum_{j=0}^\infty F T^j U \subseteq \sum_{j=0}^\infty W_j = W$, as desired.
\end{proof}

\section{Algebras and sets of derivations}

In this section, an {\it algebra} over a field $\bk$ is a pair $(B,\cdot)$ where $B$ is a $\bk$-vector space and ``$\cdot$'' is
an arbitrary $\bk$-bilinear map $B \times B \to B$, $(x,y) \mapsto x \cdot y$.

\begin{definition}  \label {0bg93mm5k6gnndrtuiecj8}
Let $B$ be an algebra over a field $\bk$.
\begin{enumerate}[\rm(a)]

\item A {\it $\bk$-derivation} of $B$ is a $\bk$-linear map $D : B \to B$ satisfying
$D(x \cdot y) = D(x) \cdot y + x \cdot D(y)$ for all $x,y \in B$. 
The set of all $\bk$-derivations of $B$ is denoted $\Der_\bk(B)$ and is a Lie algebra over $\bk$, with Lie bracket defined 
by $[D,E] = D \circ E - E \circ D$ for $D,E \in \Der_\bk(B)$.

\item We say that a subset of $\Der_\bk(B)$
is {\it locally finite\/} (resp.\ {\it weakly locally finite\/}) if it is so when viewed as a set of $\bk$-linear endomorphisms of $B$.

\item A derivation $D \in \Der_\bk(B)$ is {\it locally finite\/} if $\{D\}$ is a locally finite subset of $\Der_\bk(B)$.
We use the notation $\lfd_\bk(B) = \setspec{ D \in \Der_\bk(B) }{ \text{$D$ is locally finite} }$ 
and note that the set $\lfd_\bk(B)$ is in general neither closed under addition nor under the Lie bracket.

Clearly, if $\Delta$ is a weakly locally finite subset of $\Der_\bk(B)$ then $\Delta \subseteq \lfd_\bk(B)$.

\item The $\bk$-algebra $B$ is {\it derivation-finite\/} if there exists a finite subset $S$ of $B$ such that:
\begin{equation*}
\text{$(*)$ \quad the only $D \in \Der_\bk(B)$ that satisfies $D(x)=0$ for all $x \in S$ is the zero derivation.}
\end{equation*}

\end{enumerate}
\end{definition}

\begin{nothing*} \label {E846360A-FB67-4DC0-8E96-92430F4D1817}
{\bf Examples of derivation-finite $\bk$-algebras.}
\begin{enumerate}[\rm(a)]

\item {\it If $\bk$ is any field and $B$ is any type of finitely generated $\bk$-algebra (for instance a finitely generated commutative $\bk$-algebra,
or a finitely generated Lie algebra over $\bk$) then $B$ is derivation-finite.}
Indeed, any finite generating set of $B$ satisfies condition $(*)$ of Definition \ref{0bg93mm5k6gnndrtuiecj8}(d).

\item {\it If $A$ is a finitely generated commutative algebra over a field $\bk$, then every localization of $A$ is a derivation-finite $\bk$-algebra.}
To see this, let $\{a_1, \dots, a_n\}$ be a generating set of $A$ and $T$ a multiplicative set of $A$, and consider the finite subset
$S = \{ a_1/1 , \dots, a_n/1 \}$ of $B=T^{-1}A$. Suppose that $D \in \Der_\bk(B)$ is such that $D(s)=0$ for all $s \in S$.
Since $\ker(D)$ is a subalgebra of $B$ that contains $S$, we have $D(\alpha/1)=0$ for all $\alpha \in A$.
If $b \in B$ then there exist $t \in T$ and $\alpha \in A$ such that $(t/1)b = \alpha/1$; since $D(t/1)=0=D(\alpha/1)$,
we have $(t/1) D(b) = 0$, so $D(b)=0$ because $t/1$ is a unit of $B$.
This shows that $D=0$, so $B$ is derivation-finite.

\item {\it If $\bk$ is a field of characteristic zero and $B$ is a commutative integral domain containing $\bk$
and of finite transcendence degree over $\bk$, then $B$ is a derivation-finite $\bk$-algebra.}
Indeed, there exists a finite subset $S$ of $B$ such that $B$ is algebraic over $\bk[S]$;
since the characteristic is zero, $\ker(D)$ is algebraically closed in $B$ for any $D \in \Der_\bk(B)$,
so $S$ satisfies condition $(*)$ of Definition \ref{0bg93mm5k6gnndrtuiecj8}(d).

\item {\it If $\bk$ is an arbitrary field and $B$ is a commutative integral domain containing $\bk$ and such that
$\Frac(B)/\bk$ is a finitely generated field extension (where $\Frac(B)$ denotes the field of fractions of $B$),
then $B$ is a derivation-finite $\bk$-algebra.}
Indeed, there exists a finite subset $S$ of $B$ such that $\Frac(B) = \bk(S)$; then $S$ satisfies $(*)$.

\item {\it If $X$ is an integral scheme of finite type over a field $\bk$ then the $\bk$-algebra $\Oscr_X(X)$ is derivation-finite.
(In particular, if $X$ is an irreducible variety over any field then $\Oscr_X(X)$ is derivation-finite.)}
Indeed, if $\xi$ denotes the generic point of $X$ then the canonical homomorphism $\Oscr_X(X) \to \Oscr_{X,\xi}$ is an injective $\bk$-homomorphism,
so $\bk \subseteq \Oscr_X(X) \subseteq \Oscr_{X,\xi}$. Since $\Oscr_{X,\xi}$ is the function field of $X$, which is a finitely generated field extension of $\bk$,
the field extension $\Frac(\Oscr_X(X))/\bk$ is finitely generated and the claim follows from (d).

\item {\it If $X$ is an open subset of the affine scheme $\Spec A$, where $A$ is a finitely generated algebra over a field $\bk$,\footnote{It is understood
that $A$ is commutative, associative and unital.} then $\Oscr_X(X)$ is a derivation-finite $\bk$-algebra.
(In particular, if $X$ is a quasi-affine, not necessarily irreducible variety over any field, then $\Oscr_X(X)$ is derivation-finite.)}

\end{enumerate}
\end{nothing*}

\begin{proof}[Proof of {\rm(f)}]
Let $B = \Oscr_X(X)$.
Given $f \in B$, let $X_f$ denote the set of $x \in X$ such that the
image of $f$ by the canonical homomorphism $\Oscr_X(X) \to \Oscr_{X,x}$ does not belong to the  maximal ideal of $\Oscr_{X,x}$.
By \cite[Ex.\ II.2.16]{Hartshorne}, $X_f$ is open in $X$, $\Oscr_X( X_f ) \isom B_f$, and the restriction map
$\Oscr_X( X ) \to \Oscr_X( X_f )$, $g \mapsto g|_{X_f}$, coincides with the localization map $\phi_f : B \to B_f$, $g \mapsto g/1$;
so $\ker\phi_f = \setspec{ g \in \Oscr_X(X) }{ \, g|_{X_f} = 0 }$.

Since $X$ is open in $\Spec A$,
for each $x \in X$ there exists $a \in A$ such that $x \in \Dscr(a) \subseteq X$,
where $\Dscr(a) = \setspec{ \pgoth \in \Spec A }{ a \notin \pgoth }$.
If $f = a|_X \in B$ then $X_f = \Dscr(a) \isom \Spec A_a$, so $B_f \isom \Oscr_X( X_f ) \isom A_a$ is a finitely generated $\bk$-algebra.

So there exist $f_1, \dots, f_n \in B$ such that
$X = \bigcup_{i=1}^n X_{f_i}$ and each $B_{f_i}$ is finitely generated. It follows that $\bigcap_{i=1}^n \ker\phi_{f_i} = 0$,
because if $g \in \bigcap_{i=1}^n \ker\phi_{f_i}$ then $g|_{X_{f_i}} = 0$ for all $i=1,\dots,n$, so $g=0$.
For each $i=1,\dots, n$, let  $S_i = \{ b_{i,j}/f_i^{n_{i,j}} \}_{ 1 \le j \le p_i }$ be a finite generating set of $B_{f_i}$,
where $b_{i,j} \in B$ and $n_{i,j} \in \Nat$ for all $i,j$.
Consider the finite subset $S = \{f_1,\dots,f_n\} \cup \bigcup_{i=1}^n \{ b_{i,j} \}_{ 1 \le j \le p_i }$ of $B$.
Suppose that $D \in \Der_\bk(B)$ is such that $D(s)=0$ for all $s \in S$.
For each $i=1,\dots,n$, localization gives $D_{f_i} \in \Der_\bk(B_{f_i})$ such that $D_{f_i}(s)=0$ for all $s \in S_i$,
so $D_{f_i}=0$. If $b \in B$ then $0 = D_{f_i}(b/1)=D(b)/1 = \phi_{f_i}(D(b))$ for all $i$, so $D(b) \in \bigcap_{i=1}^n \ker\phi_{f_i} = \{0\}$,
which shows that $D=0$. So $B$ is derivation-finite.
\end{proof}

\begin{remarks} \label {142CBB6B-6351-448D-8F1A-49C84484D46B}
Let $B$ be a derivation-finite algebra over a field $\bk$.
\begin{enumerate}[\rm(a)]

\item {\it If $\Lscr$ is a locally finite  Lie subalgebra of $\Der_\bk(B)$ then $\Lscr$ is a finite dimensional vector space.}

To see this, choose a finite subset $S$ of $B$ that satisfies condition $(*)$ of Definition \ref{0bg93mm5k6gnndrtuiecj8}(d).
Since $\Lscr$ is locally finite, there exists a finite dimensional $\bk$-linear subspace $U$ of $B$ such that $S \subseteq U$ and $D(U) \subseteq U$ for
all $D \in \Lscr$.  The restriction map $\Lscr \to \Leul(U)$, $D \mapsto D|_U : U \to U$, is an injective linear map,
and since $\Leul(U)$ is finite-dimensional, so is $\Lscr$.
(This generalizes \cite[Lemma 1.6.2]{KraftZaidenberg2024}.)

\item {\it If $\Lscr$ is a weakly locally finite Lie subalgebra of $\Der_\bk(B)$ then every finitely generated Lie subalgebra of $\Lscr$
is locally finite.}

Indeed, if $\Delta$ is a finite subset of $\Lscr$ then $\Delta$ is locally finite,
so Remark \ref{03442952-E213-4085-B07A-FABC7408E85A} implies that $\LieAlg{\Delta}$ is locally finite.

\item {\it If $\dim_\bk(B)$ is countable then
each weakly locally finite Lie subalgebra of $\Der_\bk(B)$ is the union of a countable increasing sequence
$\Lscr_0 \subseteq \Lscr_1 \subseteq \Lscr_2 \subseteq \cdots$ of locally finite Lie subalgebras of $\Der_\bk(B)$.}

To see this, choose a finite subset $S$ of $B$ that satisfies condition $(*)$ of Definition \ref{0bg93mm5k6gnndrtuiecj8}(d), and let $n = |S|$.
Since $\Der_\bk(B) \to B^n$, $D \mapsto (Ds)_{s \in S}$, is an injective $\bk$-linear map, the dimension of $\Der_\bk(B)$ is countable.
If $\Lscr$ is a weakly locally finite Lie subalgebra of $\Der_\bk(B)$ then the dimension of $\Lscr$ is countable,
so we can choose a family $\{ D_i \}_{ i \in \Nat }$ of elements of $\Lscr$ that spans $\Lscr$ as a vector space.
Define $\Lscr_n = \LieAlg{ D_1,\dots,D_n }$ for each $n \in \Nat$, then $\Lscr = \bigcup_{n \in \Nat}\Lscr_n$
and each $\Lscr_n$ is locally finite by (b).

\end{enumerate}
\end{remarks}

\begin{remark} \label {AE2F4C43-C9A8-4A99-805E-F255752D6D30}
Assume that $\bk$ is an algebraically closed field of characteristic zero and that $X$ is an irreducible affine $\bk$-variety.
Then one can consider the subalgebra $\Lie( \Aut X )$  of $\Der_\bk( \Oscr_X(X) )$,
and Definition 2.5 of \cite{Arz_Zaid_2026} states that a Lie subalgebra of $\Lie( \Aut X )$ is called {\it locally integrable\/} if it admits
a countable ascending filtration by locally finite subalgebras.  We claim that if $\Lscr$ is a Lie subalgebra of $\Der_\bk( \Oscr_X(X) )$ then
$$
\text{$\Lscr$ is weakly locally finite $\iff$ $\Lscr$ is a locally integrable Lie subalgebra of $\Lie( \Aut X )$.}
$$
Indeed, $(\Leftarrow)$ is obvious. 
Conversely, if $\Lscr$ is weakly locally finite then $\Lscr \subseteq \lfd_\bk( \Oscr_X(X) ) \subseteq \Lie( \Aut X )$,
where the last inclusion is \cite[Cor.\ 7.6.4]{FurterKraft}.
Since $\Oscr_X(X)$ is derivation-finite by either one of parts (a) or (e) of \ref{E846360A-FB67-4DC0-8E96-92430F4D1817},
and $\dim_\bk\Oscr_X(X)$ is countable because $\Oscr_X(X)$ is finitely generated,
\ref{142CBB6B-6351-448D-8F1A-49C84484D46B}(c) implies that $\Lscr$ is locally integrable.
\end{remark}

\smallskip

Throughout the remainder of this section, $\bk$ denotes an arbitrary field and $B$ a derivation-finite $\bk$-algebra,
where ``$\bk$-algebra'' is understood in the general sense defined at the beginning of the section.
No assumptions on $B$ and $\bk$ are imposed beyond derivation-finiteness.

\smallskip

Recall that if $D : B \to B$ is a $\bk$-derivation then $\ad(D) : \Der_\bk(B) \to \Der_\bk(B)$ is the $\bk$-linear map $E \mapsto [D,E]$, $E \in \Der_\bk(B)$.
The following result is analogous to \cite[Thm 5]{Skutin_LND_2021} and recovers \cite[Lemma 1.7.1]{KraftZaidenberg2024} as a special case.

\begin{lemma} \label {4979E5BB-25FA-42A9-9B2E-47D08C07D92D}
Let $B$ be a derivation-finite algebra over a field $\bk$,
let $\Delta$ be a subset of $\Der_\bk(B)$
and consider the subset $\ad(\Delta) = \setspec{ \ad(D) }{ D \in \Delta }$ of $\Leul( \Der_\bk B )$.
\begin{enumerate}[\rm(a)]

\item If $\Delta$ is locally finite then $\ad(\Delta)$ is a locally finite subset of $\Leul( \Der_\bk B )$.

\item \label {4A1796C5-40AF-4BA4-BFC0-392A69D677C7}
If $\Delta$ is weakly locally finite then $\ad(\Delta)$ is a weakly locally finite subset of $\Leul( \Der_\bk B )$.

\end{enumerate}
\end{lemma}

\begin{proof}
To prove (a), suppose that $\Delta$ is locally finite.
Let $E \in \Der_\bk(B)$.
Define subsets $X_0,X_1,X_2,\dots$ of $\Der_\bk(B)$ by declaring that $X_0 = \{E\}$ and $X_{i+1} = [\Delta,X_i]$ for all $i \in \Nat$,
and let $X = \bigcup_{ i \in \Nat }X_i$. 
Since $E \in X$ and $(\ad D)(X) \subseteq X$ for all $D \in \Delta$,
it suffices to show that $\overline X = \Span_\bk(X)$ is a finite dimensional subspace of $\Der_\bk(B)$.
Since $B$ is derivation-finite, we can choose a finite subset $S$ of $B$ such that
\begin{equation} \tag{$*$}
\text{the only $D \in \Der_\bk(B)$ that satisfies $D(x)=0$ for all $x \in S$ is the zero derivation.}
\end{equation}
Since $\Delta$ is locally finite, there exists a finite dimensional subspace $U$ of $B$ such that $S \subseteq U$ and $\Delta U \subseteq U$.
Since $E(U)+U$ is finite dimensional and $\Delta$ is locally finite,
there exists a finite dimensional subspace $W$ of $B$ such that $E(U)+U \subseteq W$ and $\Delta W \subseteq W$.
We have $X_0(U) = E(U) \subseteq W$. Suppose that $n \in \Nat$ is such that $X_n(U) \subseteq W$.
Then $X_{n+1}U  = [\Delta,X_n] U \subseteq \Delta X_n U + X_n \Delta U \subseteq \Delta  W + X_n U \subseteq W$.
By induction, we get $X_{n}(U) \subseteq W$ for all $n \in \Nat$, so $X(U) \subseteq W$ and hence $\overline X U \subseteq W$.
Consider the $\bk$-linear map $\Phi : \overline X \to \Leul(U,W)$, where given $D \in \overline X$ we define $\Phi(D) = D|_U : U \to W$.
Since $S \subseteq U$, condition $(*)$ implies that $\Phi$ is injective.
Since $\Leul(U,W)$ is finite dimensional, so is $\overline X$. This proves (a), and (b) is an obvious consequence of (a).
\end{proof}

\begin{lemma} \label {104779FE-B0E2-4AD0-BB60-6BC97D307EB9}
Let $B$ be a derivation-finite algebra over a field $\bk$ and
let $\Delta$ be a weakly locally finite subset of $\Der_\bk(B)$.
If $D \in \lfd_\bk(B)$ and $[D,\Delta] \subseteq \Delta$, then $\{D\} \cup \Delta$ is weakly locally finite.
\end{lemma}

\begin{proof}
By Remark \ref{03442952-E213-4085-B07A-FABC7408E85A}, $\bar \Delta = \Span_\bk(\Delta)$ is weakly locally finite.
Moreover, $[D,\Delta] \subseteq \Delta$ implies $[D,\bar\Delta] \subseteq \bar\Delta$.
So we may (and shall) assume that $\Delta$ is a linear subspace of $\Der_\bk(B)$.

Let $S$ be a finite subset of $\Delta$. We have to show that $\{D\} \cup S$ is locally finite.
By Lemma \ref{4979E5BB-25FA-42A9-9B2E-47D08C07D92D}(a), the linear map  $\ad(D) : \Der_\bk(B) \to \Der_\bk(B)$ is locally finite.
Since $[D,\Delta] \subseteq \Delta$, $\ad(D)$ maps $\Delta$ into itself; so the linear map $\ad(D) : \Delta \to \Delta$ is locally finite.
This implies that there exists a finite dimensional subspace $\Delta'$ of $\Delta$ such that $S \subseteq \Delta'$ and $(\ad D)(\Delta') \subseteq \Delta'$,
i.e., $[D,\Delta'] \subseteq \Delta'$.
If $\Bscr$ is a basis of $\Delta'$ then $\Bscr$ is a finite subset of $\Delta$, so $\Bscr$ is locally finite;
by Remark \ref{03442952-E213-4085-B07A-FABC7408E85A}, it follows that $\Span_\bk(\Bscr) = \Delta'$ is a locally finite subset of $\Leul(B)$.
So Lemma \ref{BB8F3187-088A-4A77-9DA8-3CF3D585A95B} implies that $\{D\} \cup \Delta'$ is locally finite, 
and hence that $\{D\} \cup S$ is locally finite. This proves that $\{D\} \cup \Delta$ is weakly locally finite.
\end{proof}

The following trivial observation is useful in conjunction with Zorn's Lemma:

\begin{lemma}  \label {DEA84E5E-79BC-4D4B-BF2D-C5735262EA6D}
Let $B$ be an algebra over a field $\bk$ and let $\Xscr$  be a nonempty collection of subsets of $\Der_\bk(B)$ that is totally ordered by inclusion.
If each element of $\Xscr$ is a weakly locally finite subset of $\Der_\bk(B)$ then so is $\bigcup_{X \in \Xscr} X$.
\end{lemma}

\medskip

\noindent{\bf Proof of the Theorem.}
A straightforward proof by induction shows that (a) is true.
To prove (b), suppose that $n>0$ is such that $S_n$ is weakly locally finite.
Consider the set $\Sigma$ of all sets $X$ such that $S_{n} \subseteq X \subseteq S_{n-1}$ and $X$ is weakly locally finite.
By Zorn's Lemma and Lemma \ref{DEA84E5E-79BC-4D4B-BF2D-C5735262EA6D}, there exists a maximal element $X_*$ of $\Sigma$.
For any $D \in S_{n-1}$, we have $[D,X_*] \subseteq [D,S_{n-1}] \subseteq S_{n} \subseteq X_*$;
since $X_*$ is weakly locally finite and $D \in S_{n-1} \subseteq S \subseteq \lfd_\bk(B)$,
Lemma \ref{104779FE-B0E2-4AD0-BB60-6BC97D307EB9} implies that $\{D\} \cup X_*$ is weakly locally finite,
so $\{D\} \cup X_* \in \Sigma$ and hence $D \in X_*$ by maximality of $X_*$.
This shows that $S_{n-1} \subseteq X_*$ and hence that $S_{n-1}$ is weakly locally finite.
We proved that if $S_{n}$ is weakly locally finite then so is $S_{n-1}$.
It follows by induction that $S = S_0$ is weakly locally finite, as desired.
\hfill\qedsymbol

\bigskip

\noindent{\bf Proof of Corollary \ref{21AFB0A0-0700-4E7F-8AE2-D3557F9656CC}.}
(a) Let $S = \Lscr$, define $(S_i)_{i \in \Nat}$ as in the Theorem,
and note that $S_i \subseteq \Lscr^{(i)}$ for all $i \in \Nat$.
Choose $n$ such that $\Lscr^{(n)}$ is weakly locally finite;
then $S_n$ is weakly locally finite, so $\Lscr$ is weakly locally finite by part (b) of the Theorem.
This proves (a).
In part (b) there exists $n \in \Nat$ such that $\Lscr^{(n)} = \{0\}$, so (b) follows from (a).
\hfill\qedsymbol

\bigskip

\noindent{\bf Proof of Corollary \ref{CFBD926C-B48D-4238-8595-F09AED548518}.}
Let $S$ denote the generating set mentioned in the statement.
By part (b) of the Theorem, $S$ is weakly locally finite; by Remark \ref{03442952-E213-4085-B07A-FABC7408E85A}, so is $\Lscr$.
\hfill\qedsymbol

\bigskip

\noindent{\bf Proof of Corollary \ref{CA556DEE-FE7D-4364-A217-A983EBE59700}.}
Note that $S = \Delta \cup \Lscr$ is a subset of $\lfd_\bk(B)$ that satisfies $[S,S] \subseteq \Lscr$;
in particular, $S$ is closed under the Lie bracket.
Define $(S_i)_{i \in \Nat}$ as in the Theorem.
Then $S_1 \subseteq \Lscr$, so $S_1$ is a weakly locally finite set.
By part (b) of the Theorem, $S$ is a weakly locally finite set,
so Remark \ref{03442952-E213-4085-B07A-FABC7408E85A} implies that $\Mscr = \LieAlg{S}$ is weakly locally finite.
Since $[S,S] \subseteq \Lscr$, we obtain $[\Mscr,\Mscr] \subseteq \Lscr$.
\hfill\qedsymbol

\bigskip

\noindent{\bf Proof of Corollary \ref{8B9F8C8D-5E7E-44EA-8762-9426050B1F90}.}
By \ref{E846360A-FB67-4DC0-8E96-92430F4D1817}(f), the $\bk$-algebra $\Oscr_X(X)$ is derivation-finite,
so part (b) of Corollary \ref{21AFB0A0-0700-4E7F-8AE2-D3557F9656CC} implies that $\Lscr$ is weakly locally finite.
Let $X = \bigcup_{i=1}^n U_i$ be a finite covering by affine open sets.
For each $i$, $\Oscr_X(U_i)$ is a finitely generated $\bk$-algebra and hence a $\bk$-vector space of countable dimension.
Since $\Oscr_X(X) \to \prod_{i=1}^n \Oscr_X(U_i)$, $f \mapsto (f|_{U_i})_{i=1}^n$ is an injective $\bk$-linear map,
$\Oscr_X(X)$ has countable dimension. By \ref{142CBB6B-6351-448D-8F1A-49C84484D46B}(c),
$\Lscr$ admits a countable increasing filtration by locally finite Lie subalgebras.
\hfill\qedsymbol

\section{An example}

The aim of this section is to prove the following statement, which shows that if we do not assume that $B$ is derivation-finite then
a finitely generated Lie subalgebra of $\Der_\bk(B)$ may consist of locally finite derivations without being weakly locally finite:

\begin{proposition} \label {07BB2397-08C6-4943-A81A-2FC5730BAD29}
Let $B$ be the commutative polynomial ring in $| \Nat |$ variables over a field $\bk$.
For each integer $m\ge2$, there exists a Lie subalgebra $\Lscr_m$ of $\Der_\bk(B)$ with the following properties.
\begin{enumerate}[\rm(a)]

\item $\Lscr_m$ is $m$-generated as a Lie algebra.

\item $\Lscr_m \subseteq \lnd(B) \subseteq \lfd_\bk(B)$ and $\Lscr_m$ is not a weakly locally finite subset of $\Der_\bk(B)$.

\item Every $(m-1)$-generated Lie subalgebra of $\Lscr_m$ is a nilpotent Lie algebra,
a locally nilpotent subset of $\Der_\bk(B)$, and a locally finite subset of $\Der_\bk(B)$.

\end{enumerate}
\end{proposition}

Here, $\lnd(B)$ denotes the set of locally nilpotent derivations of $B$.
A subset $\Delta$ of $\Der_\bk(B)$ is {\it locally nilpotent\/} if, for each $x \in B$ and each infinite sequence $(D_1,D_2,\dots)$ of 
elements of $\Delta$, there exists $n>0$ such that $(D_n \circ \cdots \circ D_2 \circ D_1)(x)=0$.
Following \cite[Def.\ 2.1]{Daigle_LNSets_2020}, we say that $\Delta$ is {\it uniformly locally nilpotent\/} if 
for each $x \in B$ there exists $n>0$ such that $(D_n \circ \cdots \circ D_1)(x)=0$ for all $(D_1,\dots,D_n) \in \Delta^n$.

\begin{remark} \label {3142A156-11E1-4BE7-BB0C-BA4E9224347D}
{\it Let $B$ be an algebra over a field $\bk$.
If $\Lscr$ is a finitely generated Lie subalgebra of $\Der_\bk(B)$ and a locally nilpotent subset of $\Der_\bk(B)$,
then it is a locally finite subset of $\Der_\bk(B)$.}
Indeed, let $G$ be a finite generating set of $\Lscr$.
Then \cite[Lemma 2.2]{Daigle_LNSets_2020} implies that $G$ is uniformly locally nilpotent;
so, for each $x \in B$, the set $M_x = \setspec{  (D_r \circ \cdots \circ D_1)(x) }{ r \ge 0 \text{ and } D_1, \dots, D_r \in G }$ is finite
and consequently $\Span_\bk(M_x)$ is a finite dimensional subspace of $B$ that contains $x$ and is stable under $G$.
This shows that $G$ is locally finite, and so is $\Lscr$ by Remark \ref{03442952-E213-4085-B07A-FABC7408E85A}.
\end{remark}

\begin{remark}  \label {8707A8D6-41F8-4B7E-9EEC-460E290C67C5}
Let $B$ be an algebra over a field $\bk$.
Given a subset $\Delta$ of $\Der_\bk(B)$, let $\locfin( \Delta, B )$ denote the set of elements $x \in B$
for which there exists a finite dimensional $\bk$-subspace $U$ of $B$ such that $x \in U$ and $D(U) \subseteq U$ for all $D \in \Delta$.
It is straightforward to verify that $\locfin( \Delta, B )$ is a $\bk$-subalgebra of $B$.
Clearly, $\Delta$ is a locally finite subset of $\Der_\bk(B)$ if and only if $\locfin(\Delta, B)=B$.
\end{remark}

\begin{lemma} \label {pcv0b9i349r09f}
Let $V$ be a vector space over a field $\bk$ and $B$ a commutative polynomial ring over $\bk$ satisfying $\trdeg_\bk(B) = \dim_\bk(V)$.
Then there exist an injective $\bk$-linear map $\nu : V \to B$
and an injective homomorphism of Lie algebras $\psi : \Leul_\bk(V) \to \Der_\bk(B)$ such that $\bk[ \image\nu] = B$ and 
\begin{equation}  \label {0c2i3urv98703u4t8934jgs9}
\text{for every $F \in \Leul_\bk(V)$, the diagram} \qquad
\raisebox{7mm}{\xymatrix{
B  \ar[r]^{ \psi(F) }  &  B  \\
V \ar[u]^\nu \ar[r]_F  &   V \ar[u]_\nu
}}
\qquad \text{commutes.}
\end{equation}
Moreover, the following statements are true for every subset $\Delta$ of $\Leul_\bk(V)$:
\begin{enumerate}[\rm(a)]

\item $\Delta$ is a locally nilpotent (resp.\ uniformly locally nilpotent) subset of $\Leul_\bk(V)$
if and only if $\psi(\Delta)$ is a locally nilpotent (resp.\ uniformly locally nilpotent) subset of $\Der_\bk(B)$.

\item $\Delta$ is a locally finite (resp.\ weakly locally finite) subset of $\Leul_\bk(V)$
if and only if $\psi(\Delta)$ is a locally finite (resp.\ weakly locally finite) subset of $\Der_\bk(B)$.

\end{enumerate}
\end{lemma}

\begin{proof}
Except for assertion (b), this is \cite[Lemma 5.1]{Daigle_LNSets_2020}.
To prove (b), consider a subset $\Delta$ of $\Leul_\bk(V)$.
Given $\bk$-linear subspaces $U \subseteq V$ and $W \subseteq B$, write
$\Delta U = \setspec{ F(u) }{ F \in \Delta \text{ and } u \in U }$
and $\psi(\Delta) W = \setspec{ D(w) }{ D \in \psi(\Delta) \text{ and } w \in W }$.
Then the fact that \eqref{0c2i3urv98703u4t8934jgs9} commutes implies that
$\nu\big( \Delta U \big) = \psi(\Delta) \nu(U)$ for every subspace $U \subseteq V$.

If $\Delta$ is locally finite then  $V$ is a union of finite dimensional subspaces $U$ of $V$ satisfying $\Delta U \subseteq U$.
For each such $U$ we have $\psi(\Delta) \nu(U) = \nu\big( \Delta U \big) \subseteq \nu(U)$,
so $\nu(V)$ is a union of finite dimensional subspaces $W$ of $B$ satisfying $\psi(\Delta) W \subseteq W$,
i.e., $\nu(V) \subseteq \locfin( \psi(\Delta), B )$.
Since $B = \bk[ \nu(V) ]$ and (by Remark \ref{8707A8D6-41F8-4B7E-9EEC-460E290C67C5}) $\locfin( \psi(\Delta) , B )$ is a subalgebra of $B$,
it follows that $\locfin( \psi(\Delta) , B ) = B$ and hence that $\psi(\Delta)$ is locally finite.
It is quite clear that this argument is reversible, so $\Delta$ is locally finite  if and only if $\psi(\Delta)$ is.
It then easily follows that  $\Delta$ is weakly locally finite  if and only if $\psi(\Delta)$ is.
This proves (b).
\end{proof}

\begin{nothing*}{\bf Proof of Proposition \ref{07BB2397-08C6-4943-A81A-2FC5730BAD29}.}
The following statement is valid over an arbitrary field $\bk$ and follows from Golod's famous counter-example to the Kurosh problem
(\cite{Golod68}, cited in \cite[Thm 6.2.9]{RowenBookVol2}):
\begin{quote} 
\it
For each integer $m \ge2$, there exists an infinite dimensional associative $\bk$-algebra $A$ generated by $m$ elements
such that every $(m-1)$-subset of $A$ is nilpotent.\footnote{A subset $H$ of $A$ is \textit{nilpotent} if there 
exists $n>0$ such that $h_n \cdots h_1=0$ for all $(h_1, \dots, h_n) \in H^n$.}
Moreover, $\bigcap_{n=1}^\infty A^n = \{0\}$.
\end{quote}
Let $\bk$ be a field, let $m\ge2$, and choose a $\bk$-algebra $A$ as in the above statement.
Let $\phi : A \to \Leul_\bk(A)$ be the $\bk$-linear map that sends each $a \in A$ to the $\bk$-linear map $\phi(a) : A \to A$, $x \mapsto ax$.
Since $A$ is associative, $\phi(a_1a_2) = \phi(a_1) \circ \phi(a_2)$ for all $a_1,a_2 \in A$, i.e., $\phi$ is a homomorphism of associative algebras
and $\phi(A)$ is an associative subalgebra of $\Leul_\bk(A)$.
Let $\{g_1,\dots,g_m\}$ be a generating set of $A$.
Then 
$$
A \ = \ \sum_{i=1}^m \bk g_i \ + \!\! \sum_{ 1 \le i,j \le m} \!\!\! \bk g_i g_j + \cdots \ \  =  \ \ \sum_{i=1}^m \bk g_i + \sum_{i=1}^m A g_i ,
$$
so the fact that $A$ is infinite dimensional implies that there exists $j \in \{1,\dots,m\}$ such that $A g_j$ is infinite dimensional.
Since $\setspec{ F(g_j) }{ F \in \phi(A) } = \setspec{ [\phi(a)](g_j) }{ a \in A } = Ag_j$, this shows that $\phi(A)$ is not a locally finite subset of $\Leul_\bk(A)$.
Since $\phi(A) = \AssAlg{\phi(g_1),\dots,\phi(g_m)}$,
Remark \ref{03442952-E213-4085-B07A-FABC7408E85A} implies that  $\{\phi(g_1),\dots,\phi(g_m)\}$ is not locally finite.
So
\begin{equation}  \label {A78830BE-0722-4BFF-A60D-6A27672F3A08}
\begin{minipage}{.7\textwidth}
the Lie subalgebra $L := \LieAlg{\phi(g_1),\dots,\phi(g_m)}$ of $\Leul_\bk(A)$ is not a weakly locally finite subset of $\Leul_\bk(A)$.
\end{minipage}
\end{equation}
Let $B$ be the commutative polynomial ring in $| \Nat |$ variables over $\bk$.
Since $A$ is infinite dimensional and finitely generated, we have $\dim_\bk(A) = |\Nat| = \trdeg_\bk(B)$, so Lemma \ref{pcv0b9i349r09f} implies that 
there exist an injective $\bk$-linear map $\nu : A \to B$
and an injective homomorphism of Lie algebras $\psi : \Leul_\bk(A) \to \Der_\bk(B)$ such that $\bk[ \image\nu] = B$ and diagram (\ref{983gh5vbbmBvnd2736t4udhi}-i) commutes 
for all $F \in \Leul_\bk(A)$.
\begin{equation}  \label {983gh5vbbmBvnd2736t4udhi}
\mbox{(i)} \quad \xymatrix{
B  \ar[r]^{ \psi(F) }  &  B  \\
A \ar[u]^\nu \ar[r]^F  &   A \ar[u]_\nu
}
\qquad
\mbox{(ii)} \quad \xymatrix@R=12pt{
A \ar@{->>}[dr]  \ar[r]^-{\phi}   &    \Leul_\bk(A)  \ar[r]^-{\psi} & \Der_\bk(B)    \\
 &   \phi(A) \ar@{^{(}->}[u] \\
& \ar@{^{(}->}[u] L \ar[r]^-{\isom} &  \Lscr_m \ar@{^{(}->}[uu] 
}
\end{equation}
Consider the Lie subalgebra $\Lscr_m := \psi(L)$ of $\Der_\bk(B)$, as in the commutative diagram (\ref{983gh5vbbmBvnd2736t4udhi}-ii),
and note that $\psi|_L : L \to \Lscr_m$ is an isomorphism of Lie algebras.
Since $L$ is $m$-generated, so is $\Lscr_m$.
By \eqref{A78830BE-0722-4BFF-A60D-6A27672F3A08} and part (b) of Lemma \ref{pcv0b9i349r09f}, $\Lscr_m$ is not weakly locally finite.

Note that $\Lscr_m$ is the Lie algebra that is denoted ``$L$'' in Example 5.6 of \cite{Daigle_LNSets_2020}.
So from that Example we obtain that $\Lscr_m \subseteq \lnd(B)$ 
and that every $(m-1)$-generated Lie subalgebra $\Lscr$ of $\Lscr_m$ is a nilpotent Lie algebra and a locally nilpotent subset of $\Der_\bk(B)$;
by Remark \ref{3142A156-11E1-4BE7-BB0C-BA4E9224347D} it follows that $\Lscr$ is also a locally finite subset of $\Der_\bk(B)$,
so the Proposition is proved. \hfill\qed
\end{nothing*}

\medskip

\noindent {\bf Funding.} The first author is supported by the University of Padova through the ``BIRD 2024 Project'' research grant.

%
%

\providecommand{\bysame}{\leavevmode\hbox to3em{\hrulefill}\thinspace}
\providecommand{\MR}{\relax\ifhmode\unskip\space\fi MR }
\providecommand{\MRhref}[2]{%
  \href{http://www.ams.org/mathscinet-getitem?mr=#1}{#2}
}
\providecommand{\href}[2]{#2}


\end{document}